\documentclass[11pt]{amsart}

\usepackage{mathrsfs}
\usepackage{amssymb}
\usepackage{amsthm}
\usepackage{dsfont}
\usepackage{amssymb}
\usepackage{amsmath,amscd}
\usepackage[all,cmtip]{xy}
\usepackage{mathrsfs}
\usepackage{multirow}
\usepackage{hyperref}
\usepackage{graphicx}
\usepackage{amsfonts, amsthm}
\usepackage{fullpage}
\usepackage{enumerate}
\usepackage{appendix}
\usepackage{cite}

\normalfont\upshape

\theoremstyle{plain}
\newtheorem{thm}[subsection]{Theorem}
\newtheorem{lemma}[subsection]{Lemma}
\newtheorem{prop}[subsection]{Proposition}
\newtheorem{cor}[subsection]{Corollary}

\theoremstyle{definition}
\newtheorem{rmk}[subsection]{Remark}
\newtheorem{defn}[subsection]{Definition}
\newtheorem{example}[subsection]{Example}

\theoremstyle{remark}
\newtheorem{Claim}{Claim}

\let\a\alpha
\let\b\beta

\let\e\epsilon

\let\g\gamma

\let\l\lambda
\let\m\mu

\let\na\nabla
\let\p\psi
\let\o\omega
\let\r\rho
\let\s\sigma

\let\th\theta

\let\z\zeta
\def\scr{\mathscr}

\let\D\Delta

\let\G\Gamma

\let\O\Omega

\newcommand{\ra}{\longrightarrow}
\newcommand{\xa}{\xrightarrow}

\newcommand{\bGn}{G[p^n]\otimes_{\Z} H[p^n]}
\newcommand{\ro}{{\mbox{ or }}}

\newcommand{\stab}{\operatorname{stab}}
\newcommand{\im}{{\rm im}\:}

\newcommand{\DD}{{\mathbb D}}
\newcommand{\GG}{{\mathbb G}}

\newcommand{\Z}{{\mathbb Z}}
\newcommand{\C}{{\mathbb C}}
\newcommand{\Q}{{\mathbb Q}}
\newcommand{\R}{{\mathbb R}}

\newcommand{\cA}{\mathcal{A}}

\newcommand{\cF}{\mathcal{F}}
\newcommand{\cG}{\mathcal{G}}
\newcommand{\cE}{\mathcal{E}}

\newcommand{\cL}{{\mathcal{L}}}
\newcommand{\cO}{\mathcal{O}}
\newcommand{\cT}{\mathcal{T}}
\newcommand{\cV}{\mathcal{V}}

\newcommand{\sL}{{\mathscr L}}

\newcommand{\Spec}{{\mbox{Spec }}}

\newcommand{\cris}{{\mathrm{Cris}}}
\newcommand{\cri}{{\mathrm{cris}}}

\newcommand{\Aut}{{\mathrm{Aut}}}

\newcommand{\Rep}{{\mathrm{Rep}}}
\newcommand{\colim}{{\mathrm{colim}}}
\newcommand{\iso}{{\mathrm{Isocris}}}

\newcommand{\cha}{{\mbox{char }}}

\newcommand{\Pic}{{\text{Pic }}}

\newcommand{\bG}{{G}} 
\newcommand{\bH}{{H}} 

\newcommand{\univ}{{\mathrm{univ}}}
\newcommand{\Hom}{{\mathrm{Hom}}}

\newcommand{\Pp}{{(\s^*-\mathrm{Id})}}

\begin{document}

\title [ ]{A characterization of Mumford curves with good reduction }
\author{Jie Xia}
\maketitle

%\tableofcontents

\begin{abstract}
Mumford defines a certain type of Shimura curves of Hodge type, parameterizing polarized complex abelian fourfolds. In this paper, we study the good reduction of such a curve in positive characteristic and give a characterization in the generically ordinary case.
\end{abstract}

\section{Introduction}
\subsection{Background}
This paper aims to characterize certain Shimura varieties of Hodge type with good reduction. This description will serve as the main example in our work on defining Shimura varieties in positive characteristic. 

Let $A$ be an abelian variety over $\C$. The elements in 
\[H^{2r}(A, \Q) \cap H^{r,r}(A)\] are called Hodge classes of $A$. The Hodge group of an abelian variety $A$ is the largest $\Q-$subgroup of $GL(H^1(A, \Q))$ which leaves all Hodge classes invariant. Mumford defines in \cite{famofav}, a Shimura variety of Hodge type as a moduli scheme of abelian varieties (with a suitable level structure) whose Hodge group is contained in a prescribed Mumford-Tate group, arising from a Hermittain symmetric pair==. 

Furthermore,  Mumford exemplifies Shimura curves of Hodge type in \cite{Mum}. He constructs a simple algebraic group $Q$ over $\Q$ which is the $\Q-$form of the real algebraic group \[SU(2)\times SU(2) \times SL(2), \] a cocharacter $h$ of $Q_\R$
\begin{align*}
h: \mathbb{S}_{m}(\R) \ra& Q(\R)\\ e^{i\th} \mapsto&I_4 \otimes \begin{pmatrix} \cos \th & \sin \th \\-\sin \th & \cos \th \\ \end{pmatrix}
\end{align*} 
and an eight dimensional absolute irreducible rational representation $V$ of $Q$. The pair $(Q,h)$ is a Shimura datum and with the representation $V$, it defines Shimura curves of Hodge type, parameterizing four dimensional polarized abelian varieties over $\C$.

Generalizing the construction, one is able to define Shimura curves of Hodge type parameterizing $2^m$ dimensional polarized abelian varieties (see Section \ref{Mumford curve}). We call such Shimura curves (with its universal family) \textit{Shimura curves of Mumford type}, or for simplicity \textit{Mumford curves}, denoted as $M$ and let $A\ra M$ be a universal family of abelian varieties over $M$. 

Mumford curves play a significant role among smooth Shimura curves of Hodge type. Specifically, Theorem 0.8 in \cite{Moller} shows the universal family over a Shimura curve has a strictly maximal Higgs field. Theorem 0.5 in \cite{Zuo} shows that up to powers and isogenies, the only smooth families of abelian varieties over curves with maximal Higgs field are Mumford curves.

Rutger Noot and Kang Zuo et al.\ have studied the reduction of a fiber of a Mumford curve(\cite{Zuo2}, \cite{Rutger}). They especially classify the possible Newton polygons of such a good reduction, which we will use in this paper. Their approaches both have the flavor of p-adic Hodge theory while we mainly use crystalline cohomology and deformation theory. 
\subsection{Main Result}

\begin{defn} \label{*p}
For any prime number $p$ and integer $m$, we say the pair $(\tilde X\ra \tilde C,k)$ satisfies $(*_{p,m})$ if it satisfies the following properties:
\begin{enumerate}
\item $k$ is an algebraically closed field of characteristic $p$,
\item $\tilde C$ is a proper smooth curve over $W(k)$ and $\tilde X \ra \tilde C$ is a family of abelian varieties of dimension $2^m$ over $\tilde C$, 
\item there exists a versally deformed height 2 Barsotti-Tate (BT) group ${\tilde G}$ and a height $2^m$ etale BT group ${\tilde H}$ over $\tilde C$ such that $\tilde X[p^\infty]\cong {\tilde G} \otimes {\tilde H}:= \colim_n ({\tilde G}[p^n]\otimes {\tilde H}[p^n])$, 
\item the reduction of $\tilde X \ra \tilde C$ at $k$ is generically ordinary.
\end{enumerate}
\end{defn}

As in \cite{Xia},  a height 2 BT group ${\tilde G}$ over $\tilde{C}$ is versally deformed if the Kodaira-Spencer map $T_{\tilde{C}} \rightarrow t_{\tilde G} \otimes t_{{\tilde G}^*}$ is an isomorphism, or equivalently the Higgs field $\theta_{\tilde G}$ (see Section \ref{the BT group}) is an isomorphism. 

\begin{thm} \label{main thm 1}
Let $A \ra M$ be a Mumford curve, parameterizing principally polarized abelian varieties of dimension $2^m$. For infinitely many primes $p$, there exists a pair $(\tilde X\ra \tilde C, k)$ satisfying $(*_{p,m})$ (see Definition \ref{*p}) and 
\[(\tilde X \ra \tilde C)\otimes_{W(k)} \mathbb{C} = (A \ra M).\]
\end{thm}

\begin{rmk}
When choosing $p$, it suffices to require that \begin{enumerate}
\item $p>2$, see \ref{p>2}
\item $M$ admits a good reduction at the place $p$, see \ref{Lefschetz principle}
\item the reflex field of $M$ is splitting over $p$, see \ref{intersects with ordinary}. 
\end{enumerate}
\end{rmk}

Since the reduction $X \ra C$ of $\tilde{X}\ra \tilde{C}$ at $k$ also admits the decomposition $X[p^\infty]\cong G \otimes H$, \ref{main thm 1} provides examples for Theorem 7.4 in \cite{Xia}.

%By the \cite[Main Theorem 1]{de J}, we know the condition (3) in Definition \ref{*p} is equivalent to prove the crystalline Dieudonne module associated to $\tilde X/\tilde C$ is a tensor product of a rank 2 Dieudonne crystas and a rank $2^m$ unit root crystal. 

%The last and most difficult part is to decompose the Frobenius of $\DD(X)$ as a tensor product. Using the generic ordinary property, we can prove the Frobenius admits such a tensor product. 

\subsection{Structure of the paper}

The goal of the paper is to prove \ref{main thm 1}. In Section \ref{Mumford curve}, we introduce the basic definitions. By Lefschetz principle \ref{Lefschetz principle}, a Mumford curve with the universal family can descend to a Witt ring whose special fiber $X/C$ is smooth. The definition of Mumford curves implies the Dieudonne crystal $\DD(X/C)$ of the abelian scheme $X$ is a tensor product of $m+1$ rank 2 crystals: 
\[\DD(X)  \cong \cV_1 \otimes \cV_2 \otimes \cV_3 \cdots \otimes \cV_{m+1} .\] 

In Sections \ref{T category} and \ref{notation}, we set up some notation and basic facts of Tannakian categories. In Section \ref{some lemmas}, we prove two important lemmas (\ref{structure of Q_i}, \ref{decomposition of general F}) in context of abstract Tannakian categories. They are key ingredients in determining the Tannakian groups of the rank 2 isocrystals  and their Frobenius pullback.

In Section \ref{tensor decomposition of E}, we describe the structure of $\cV_i$ in the terminology of Tannakian categories. It is shown in \ref{the structure of G} that the Tannakian group of each isocrystals $\cV_i$ is $SL(2)$. 
Furthermore, we investigate the tensor decomposition of the Frobenius morphism $F$. Firstly, it is shown in  \ref{descent to W(k)} that $F$ can be decomposed to the tensor product of $\phi_i$ which are morphisms between rank 2 crystals.  Secondly, imposing the generically ordinary assumption permits a refinement, i.e. the permutation $s$ in \ref{W,V} fixes an index, say 1. Lastly, we adjust $\phi_1$ to be an actual Frobenius morphism of a rank 2 crystal. That requires proving that $\s^*-\rm{Id}$ on $\Pic(C/W(k)_\cri)$ is surjective. This step is in Section \ref{the surjectivity}.

Summarizing the above results in Section \ref{the crystals}, we  construct a rank 2 Dieudonne crystal $\cV$ and a rank $2^m$ unit root crystal $\cT$ , such that $\DD(\tilde X / \tilde C)\cong \cV \otimes \cT$.  We conclude the proof of \ref{main thm 1}  in Section \ref{the BT group} by studying the BT groups corresponding to $\cV$ and $\cT$. 

%In the following section, we use $\bar Y$ to denote an object defined over $k$. In particular, if $\tilde X$ is already defined over $W(k)$, then $X$ denote the special fiber of $\tilde X$.  

\subsection*{Acknowledgment}
 I would like to express my deep gratitude to my advisor A.J. de Jong, for suggesting this project to me and for his patience and guidance throughout its resolution. This paper would not exist without many inspiring discussions with him.  I also thank Vivek Pal for grammar checking.
\section{Mumford curves and their reduction} \label{Mumford curve}
We review the generalization of Mumford's construction, following \cite{Zuo}.
\subsection{Mumford curves over $\C$} 

Let $K$ be a totally real field of degree $m+1$ and $D$ be a quaternion division algebra over $K$ which splits only at one real place and $\rm{Cor}_{K/\Q}(D)=M_{2^{m+1}}(\Q)$. In this case $D\otimes_\Q \R \cong \mathbb{H} \times \cdots \times \mathbb{H} \times M_2(\R)$ and $m$ is even.

Let $\bar \ $ be the standard involution of $D$, and let 
\[Q=\{x\in D^* | x\bar x=1\}. \]
Then $Q$ is a simple algebraic group over $\Q$ which is the $\Q-$form of the real algebraic group 
\[SU(2)^{\times m} \times SL(2,\R).\] 
Since $\rm{Cor}_{K/\Q}(D)=M_{2^{m+1}}(\Q)$, $Q$ admits a natural $2^{m+1}$ dimensional rational representation $V$ whose real form is 
\[\r: SU(2)^{\times m} \times SL(2) \ra SO(2^m)\times SL(2) \text{ acting on }  \R^{2^{m+1}}. \]
Note $Q_\C=SL(2,\C)^{ \times m+1}$. Then $V_\C$ is the tensor of $m+1$ copies of standard representation $\C^2$ of $SL(2,\C)$: 
\begin{equation} \label{the tensor decomposition of V}
V_C=\C^2 \otimes \C^2 \cdots \otimes \C^2\ \ \ (m+1 \text{ factors})
\end{equation} 

Let \begin{align*}
h: \mathbb{S}_{m}(\R) \ra& Q(\R)\\ e^{i\th} \mapsto&I_{2^m} \otimes \begin{pmatrix} \cos \th & \sin \th \\-\sin \th & \cos \th \\ \end{pmatrix}.
\end{align*} 
Then $(Q,h)$ defines a Shimura datum. Generically $\r(Q)$ is the Hodge group of $V$. 

Let $\stab(h)\subset Q_\R$ be the stabilizer of $h$. Then $\stab(h)$ is a maximal compact subgroup of $Q_\R$ and hence conjugate to $SO(2) \times SU(2)^{\times m}$. So $Q_\R/\stab(h)\cong Sp(1,\R)/SO(2, \R)\cong \frak{h}$ the upper half plane. Since $\ker \r\subset \stab(h)$, we have 
\[\r(Q)/\stab(\r\circ h)=Q_\R/\stab(h)=\frak{h}.\]
Let $\G \subset Q_\R$ be an arithmetic subgroup such that $\G$ acts freely and properly discontinuous on $\frak{h}$. Note $\ker \r \subset Z(Q)$ and then it fixes $h$, $\G\hookrightarrow \r(Q(\R)).$ 

The (one-dimensional) Shimura varieties defined by $(Q,h)$ are called \textsl{Mumford curves}. With a small enough level structure $\G$, a Mumford curve is proper and smooth. It is a Shimura curve of Hodge type, parameterizing a family $A$ of polarized abelian varieties of dimension $2^m$. In particular, we can view the space $V$ as a $\Q-$local system over $M$.

As \ref{main thm 1} indicates, we study the good reduction of $A \ra M$ in this paper.
\subsection{Monodromy}
Since $\frak{h}$ is simply connected, $\pi_{1}(M)=\G$. The local system $V$ induces a monodromy $\G \ra \Aut(V_\C)$. Further, the tensor components $\C^2$ of $V_\C$ also admit  representations of $\G$ and hence also monodromy. Since $\G_\C\subset Q_\C \cong SL(2)^{\times 3}$, $\wedge^2 \C^2$ is a trivial representation of $\G$.

\begin{defn}
For any monodromy $\G \ra GL(n)$, the \textsl{algebraic monodromy group} is defined to be the Zariski closure of the image of the monodromy. The \textsl{connected algebraic monodromy group} is the connected component of the identity of the algebraic monodromy. 
\end{defn} 

\begin{prop} \label{over C}
The algebraic monodromy group induced by $\C^2$ in (\ref{the tensor decomposition of V}) is $SL(2,\C)$ and that of $V_\C$ is the image of $SL(2,\C)^{\times m+1}$ in $\Aut(V_\C)$.
\end{prop}

\begin{proof}
From above, the monodromy induced by the representation of $Q$, is tensor of $m+1$ copies of monodromies $\C^2$. Let $K_i, 1\leq i\leq n+1$ be the corresponding algebraic monodromy groups. Since $\wedge^2 \C^2$ is a trivial representation of $\G_\C$, $K_i \subset SL(2, \C)$. 

By \cite{Andre}, the connected algebraic monodromy on $V_{\Q}$ is a normal subgroup in the Hodge group $\r(Q)$. Since $Q$ is simple, $\r(Q)$ is also simple over $\Q$. Thus the connected algebraic monodromy is $\r(Q)$. Since $\r(Q)_\C=\im(SL(2,\C)^{\times m+1}\ra \Aut(V_\C))$ is connected, the connected complex  algebraic monodromy of $V_\C$ is $\im(SL(2,\C)^{\times m+1}\ra \Aut(V_\C))$. 

Note the complex algebraic monodromy of $V_\C$ is $\im(\prod_i K_i\ra \Aut(V_\C))$. Therefore 
\[\im(\prod_i K_i\ra \Aut(V_\C))^{o}=\im(SL(2,\C)^{\times m+1}\ra \Aut(V_\C)). \] Then necessarily, $K_i=SL(2,\C)$ for each $i$.
\end{proof}

\subsection{Lefschetz Principle }  \label{Lefschetz principle}
By Lefschetz Principle (see \cite{Lefschetz}), we mean the process that all the coefficients of polynomials, defining a variety of finite type over a field, generate a subring $R$ of finite type over $\Z$, such that the variety can be defined over $R$.  Note this process can be easily generalized to morphisms of finite type or vector bundles of finite rank.

Apply Lefschetz Principle to $A \ra M$ and the flat vector bundles induced by $\C^2$. We obtain these data can descend from $K$ to a ring $R$ finite type over $\Z$. Throwing away finite places, we can assume $R$ is smooth over $\Z$.  Let $k$ be a residue field of $R$ with characteristic$>2$ such that $M$ admits a good reduction over $k$. By smoothness of $R$, we have the a lifting from $\Spec W_n(k)$ to $\Spec R$:
\[\xymatrix{
\Spec k \ar[r] \ar[d] & \Spec R \ar[d]\\
\Spec W_n(k) \ar[r] \ar@{-->}[ru]& \Spec \Z
}
\]
Therefore we find a morphism $\Spec W(k) \ra \Spec R$.

Let $\tilde X \xa{\tilde \pi} \tilde C$ (resp. $\cV_i$) be the base change $A \ra M$ (resp. the flat vector bundles) from $\Spec R$ to $\Spec W(k)$. Let $X, C$ be the special fiber of $\tilde X, \tilde C$.
Let $\cE$ be the Hodge bundle $\cE=R^1\tilde \pi_*(\O^._{\tilde X/\tilde C})$ and $\cE$ admits the Gauss-Manin connection. By \cite[Theorem 6.6]{Berthelot}, the category of crystals on $C$ is equivalent to the category of modules with an integrable connection (MIC).  In particular, the Hodge bundle $\cE$ corresponds to the Dieudonne crystal $R^1\pi_{*,\cri} (\cO_{X})$. Let us denote the crystal still as $\cE$. The vector bundles $\cV_i$ also correspond to crystals and denote the corresponding crystals as $\cV_i$ as well. Then as crystals
\[\cE \cong \cV_1 \otimes \cV_2 \otimes \cV_3 \cdots \otimes \cV_{m+1}.\]

\section{Tannakian Category  } \label{T category}

In this section, we review some basic constructions and facts regarding Tannakian categories that we will need later. 

\begin{defn} Let $L$ be a field of characteristic 0.
A Tannakian category $T$ (over $L$) is a $L-$linear neutral rigid tensor abelian category with an exact fiber functor $\o: T\ra \rm{Vect}_L $. 
\end{defn}

\begin{thm} (\cite[Theorem 2.11]{Deli}) \label{Deli}
For any Tannakian category $T$, there exists an $L-$algebraic group $G$ such that $T$ is equivalent to $\Rep_L(G)$ as tensor categories. 
\end{thm}
We mainly use the following two special Tannkian categories. 
\begin{example} \label{MIC}
Choose a point $c\in M$,  the category of all MIC on $M$ with fiber functor $\cF \ra \cF_c$ form a Tannakian category. By \ref{Deli}, it corresponds to $\Rep_\C(G_{\univ})$. 

By Riemann-Hilbert correspondence, the category of MIC on $\tilde{C}$ is equivalent to $\Rep(\pi_1(M))$. The algebraic group $G_{\univ}$ can be constructed from $\pi_1(M)$ by the following: 
\[G_{\univ}=\lim H\]
 where $H$ lists the Zariski closure of image of $\pi_1(M)$ in $GL(W)$ for all complex representations $W$. Note the system of $H$ is projective. So the image of $G_{\univ} \ra \Aut(W)$ is exactly the Zariski closure of the image of $\pi_1(M)$ in $GL(W)$. 
\end{example}
Let $B(k)$ be the fraction field of $W(k)$.
\begin{example} (\cite[VI 3.1.1, 3.2.1]{Saa} ) \label{T and iso}
Inverting $p$ in the category $\cris(C/W(k))$, we obtain the category of isocrystals $\iso(C/W(k))$. Similar to \ref{MIC}, the category $\iso(C/W(k))$ forms a Tannakian category over $B(k)$, with fiber functor associated to a $k-$point of $C$.  So there exists a $B(k)-$affine group scheme $P_{\univ}$ such that the following two categories are equivalent. 
\[\{\text{finite locally free isocrystals on } C/W(k) \} \longleftrightarrow \Rep_{B(k)}(P_{\univ}).\]
An object $\cF'$ in $\iso(C/W(k))$ is called effective if it is from an object $\cF$ in $\cris(C/W(k))$,i.e. $\cF' =\cF \otimes B(k)$. For any morphism $f: \cF\otimes B(k) \ra \cG \otimes B(k)$ between effective objects in $\iso(C/W(k))$, there exists $m\in \Z$ such that $p^m f: \cF \ra \cG$ is a morphism in $\cris(C/W(k))$. 

Note different from \cite[VI 3.1.1, 3.2.1]{Saa}, $\iso(C/W(k))$ denotes just the isocrystals, not the $F-$isocrystals. So $P_{\univ}$ is an affine group scheme over $B(k)$, note over $\Q_p$.
\end{example}
 We conclude this section by a simple result. 
\begin{prop} \label{the inclusion}
For any Tannakian category $T$ and $W, V \in T$, let $<W>$ denote the Tannakian subcategory generated by $W$, with Tannakian group $G_W$. Similarly, $<V>=\Rep_k(G_V),<W,V>=\Rep_k(K)$. Then there exists a natural injection $K \hookrightarrow G_W\times G_V$.
\end{prop}

\begin{proof}
Since $W, V\in \Rep(K)$, by (\cite[2.21]{Deli}), $K$ admits surjections onto $G_W$ and $G_V$. Then $K$ admits a map $K \ra G_W\times G_V$. The induced morphism $\Rep(G_W \times G_V) \ra \Rep(K)$ satisfies (\cite[2.21(2)]{Deli}). So the map is injective. 
\end{proof}

\section{Notation} \label{notation}
We summarize the notation and fix them till the end. 
\begin{itemize} \itemsep 4pt
\item By (iso)crystals over $C$, we always mean (iso)crystals in vector bundles over the crystalline cite $\cri(C/\Z_p)$. 
\item We use subscript $C$ to denote the reduction of an object or a morphism from $\tilde{C}$ to $C$. For instance, $\cE_C$ naturally means the associated vector bundle over $C$ from the crystal $\cE$ and $F_C$ is just the restriction of the morphism $F: \cE^\s \ra \cE$ to $C$. 
\item[$B(k)$] the fractional field of $W(k)$.
\item[$Q$]the reductive group defining the Mumford curves.
\item[$\tilde X\ra \tilde C$ ] the descent of the Mumford curve with the family of abelian varieties $A/M$ to $\Spec W(k)$.
\item[$\s$] the absolute Frobenius on $C$.
\item[$\cE, \cV^\s_i, \cV_i$] $\cE=R^1\pi_{\cri,*}(\cO_{X})$, $\cE\cong \cV_1 \otimes \cV_2 \otimes \cV_3 \cdots \otimes \cV_{m+1}$ and $\cE^\s\cong \cV^\s_1 \otimes \cV^\s_2 \otimes \cV^\s_3\cdots \otimes \cV^\s_{m+1}$ . 
\item[$P_\univ$] the Tannakian group of the category of finitely locally free isocrystals on $\tilde C/W(k)$.
\item[$E,W_i,V_i$] $B(k)-$representations of $P_\univ$ corresponding to $(\cE, \cV^\s_i, \cV_i)$, respectively.
\item[$P$]the Tannakian group of the subcategory generated by $\cE$, i.e. $\im(P_\univ \ra \Aut(E))$ .
\item[$P_i$] the Tannakian group of the subcategory generated by $\cV_i$, i.e. $\im(P_\univ \ra \Aut(V_i))$.
\item[$Q_i$]  the Tannakian group of the subcategory generated by $\cV^\s_i$, i.e. $\im(P_\univ \ra \Aut(W_i))$.
\item[$Q^0_i$] the connected component of the identity in $Q_i$.
\item[$P'$] the Tannakian group of the subcategory generated by $\{\cE^\s\}$.
\item[$Q'$] $Q'=\im(P_\univ \ra \prod Q_i)$.
\item[$K_{12}$] the Tannakian group of the subcategory generated by $\{\cV^\s_1, \cV_2\}$.
\end{itemize}
\section{Some Important lemmas on Tannakian categories} \label{some lemmas}

In this section, we prove some lemmas about general Tannakian categories. These lemmas will be applied to proving \ref{main thm 1} in the next section. 
\begin{lemma} \label{dim>2}
For any $g\in GL(2)$, the centralizer $Z(g)$ of $g$ has dimension $\geq 2$ as a variety.  
\end{lemma}
\begin{proof} Let $g=\begin{pmatrix}
a & b\\
c&d
\end{pmatrix}$. The centralizer of $g$ 
\[
\begin{pmatrix}
a & b\\
c&d
\end{pmatrix}
\begin{pmatrix}
x & y\\
z & w
\end{pmatrix}=
\begin{pmatrix}
x & y\\
z & w
\end{pmatrix}
\begin{pmatrix}
a & b\\
c&d
\end{pmatrix}
\] implies \[bz=cy, (a-d)y=b(x-w).\] Note $\dim GL(2)=4$. As a subvariety of $GL(2)$, $Z(g)$ has dimension at least 2. 
\end{proof}

The Tannakian category of isocrystals on $C/W(k)$ is equivalent to $\Rep(P_\univ)$.
\begin{lemma} \label{structure of Q_i}
Let $W_i, V_i \in \Rep(P_\univ)$ be representations over $B(k)$, $1\leq i \leq n$.  Let $E\cong \otimes_i V_i$, $P_i=\im(P_\univ \ra \Aut(V_i))$ and $Q_i =\im(P_\univ \ra \Aut(W_i))$. Suppose we have that 
\[F: W_1 \otimes \cdots \otimes W_{m+1} \ra V_1 \otimes \cdots \otimes V_{m+1} \]
is an isomorphism between representations and $P_i=SL(2)$ for each $i$. Then 
\[Q_i =   GL(2)\ro SL(2)\times \m_k\text{ for some }k.\]
\end{lemma}
Let $Q'$ be the image of $P_{\univ}\ra \prod Q_i$, and then the projections $Q'\ra Q_i$ are surjective for each $i$. 

\begin{proof}
Let $P'$ be $\im(Q' \ra \Aut(E)\cong GL(2^{m+1}))$. Then we have the following commutative diagram: 
\[\xymatrix{
Q_1\times Q_2\times Q_3 \cdots \times Q_{m+1} \ar@{^{(}->}[d] & Q' \ar@{^{(}->}[l] \ar[d] &SL(2)^{\times m+1} \ar[dl]^{\text{twisted by }F}\\
GL(2)^{\times m+1} \ar[r] & GL(2^{m+1}) &
}\]
Note $SL(2)^{\times m+1} \ra GL(2^{m+1})$ is twisted by $F$. The right triangle can be specified as 
\[\xymatrix{ 
Q' \ar@{->>}[d]& SL(2)^ {\times m+1} \ar@{->>}[dl]\\
P' \ar[d] &\\
GL(2^{m+1})
}\]
where $P'$ is the common image.

Since $SL(2)$ is semisimple, so is $P'/Z(P')$. Since $\ker(Q' \ra GL(2^{m+1}))\subset\ker (GL(2)^{\times m+1} \ra GL(2^{m+1}))$, the kernel of $Q' \ra P'$ consists of just central elements. The group $Q'$ is an extension of central elements and a semisimple group. Therefore $Q'$ is reductive and $P'$ is the adjoint group of $Q'$. Further, $Q' \ra P'$ induces a morphism from the derived group $[Q',Q']$ to $P'$ which further induces a surjection to $P'/Z(P')$. 
\[[Q',Q'] \ra Q' \ra P' \ra P'/Z(P').\]

If the projection of $[Q',Q']$ to some factor $GL(2)$ has dimension less than 3, then one of the projections must have dimension 4 because of $\dim P'=3\times{2^m}$. So one of the projections would be $GL(2)$. Since the kernel of $Q'\ra P'$ is finite,  $P'$, as the image of $SL(2)^{m+1} \ra GL(2^{m+1})$ would have infinitely many centers, contradiction.

Now we have the projections of $[Q',Q']$ to each factor have precisely dimension 3. Therefore each projection has the form $SL(2)\times \m_k$.  By comparing the dimensions, $SL(2)^{\times m+1} \subset \im(Q' \ra GL(2)^{\times m+1})$. Then we have a lifting
\[SL(2)^{\times m+1} \ra [Q',Q'] \subset Q'\]
such that the right triangle is commutative
\[\xymatrix{
Q_1\times Q_2\times Q_3 \cdots \times Q_{m+1} \ar@{^{(}->}[d] & Q' \ar@{^{(}->}[l] \ar[d] &SL(2)^{\times m+1} \ar[dl]^{\text{twisted by }F} \ar@{-->}[l]\\
GL(2)^{\times m+1} \ar[r] & GL(2^{m+1}) & .
}\]
Now we classify the elements with finite kernel in $\Hom(SL(2)^{\times m+1}, GL(2)^{\times m+1})$. 

First, recall that all automorphisms of $SL(2)$ are inner and hence $\Hom(SL(2), GL(2))$ consists of the trivial morphism and the conjugation by some element in $GL(2)$. For any morphism $f\in \Hom(SL(2)^{\times m+1}, GL(2)^{\times m+1})$, restricting to each factor of $SL(2)$ gives $m+1$ inclusions $SL(2) \hookrightarrow GL(2)$. Explicitly, \begin{align*}
(g_1,1,1, \cdots)&\mapsto (\p_{11}(g_1), \p_{12}(g_1), \p_{13}(g_1), \cdots)\\
(1,g_2,1, \cdots)&\mapsto (\p_{21}(g_2), \p_{22}(g_2), \p_{23}(g_2), \cdots)\\
(1,1,g_3,\cdots)&\mapsto (\p_{31}(g_3), \p_{32}(g_3), \p_{33}(g_3), \cdots).
\end{align*}
Then $\p_{11}(g_1)$ and $\p_{21}(g_2)$ commute for any $g_i\in SL(2)$. Note all the automorphisms of $SL(2)$ are inner. So if neither of $\p_{11}$ and $\p_{12}$ is an identity, then there exists $h,k\in GL(2)$ such that $\psi_{11}, \psi_{21}$ are conjugation by $h$ and $k$, respectively. Then for any $g_i\in SL(2)$,
\begin{align}
hg_1h^{-1}kg_2k^{-1} & =kg_2k^{-1}hg_1h^{-1}\\
k^{-1}hg_1h^{-1}kg_2k^{-1}h&=g_2k^{-1}hg_1,
\end{align}
Since $Z(k^{-1}g)$ in $GL(2)$ has dimension at least 2 by \ref{dim>2}, we can choose $g_2\in SL(2)$ such that $g_2\neq \pm I$ and $g_2\in Z(k^{-1}g)$. But then from (2), $g_2$ has to commute with $g_1$, i.e. $g_2\in Z(SL(2))=\pm 1$, contradiction. Therefore at least one of $\p_{11}$ and $\p_{21}$ is identity. Further, each column $\p_{*i}$ has at least $m$ identities. So each factor $SL(2)$ is embedded into exactly one of the $m+1$ copies of $GL(2)$ and trivially to the others. 

Then $\dim Q_i=3 \ro 4$. Since $Q_i \subset GL(2)$, $Q_i = GL(2) \ro SL(2)\times \m_k$ for some integer $k$.
\end{proof}

\begin{lemma} \label{decomposition of general F}
Assumptions as \ref{structure of Q_i}, there exist a permutation $s\in S_{m+1}$, dimension 1 representation $L_i$ with $\otimes_i L_i $ trivial and isomorphisms 
\[\phi_i: W_i  \ra V_{s(i)}\otimes L_i  \]  such that \[F=\otimes_i \phi_i.\]
\end{lemma}

\subsection{Proof of \ref{decomposition of general F}}From the proof of \ref{structure of Q_i}, for each $i$, there exists a unique $P_j=SL(2)$ such that $P_j \hookrightarrow Q_i$. This inclusion is an isomorphism between $P_j$ and $Q^0_i$, which is a conjugation by some $l\in GL(2)$. Without loss of generality, assume $i=1$ and $j=2$. 

Note we have the following diagram: 
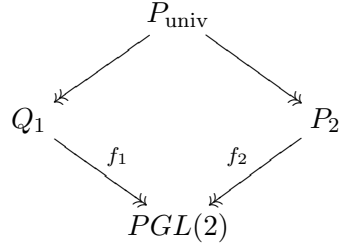
\[\xymatrix{
&P_{\univ}\ar@{->>}[dl] \ar@{->>}[dr]&\\
Q_1 \ar@{->>}[dr]^{f_1}&&P_2 \ar@{->>}[dl]_{f_2}\\
&PGL(2)&
}\] The morphism $f_1$ is the usual quotient by the center $Q_i \ra PGL(2)$. The morphism $f_2$ is twisted by the conjugation by $l$. 

\begin{Claim} this diagram is commutative.
\end{Claim}
\begin{proof} We have
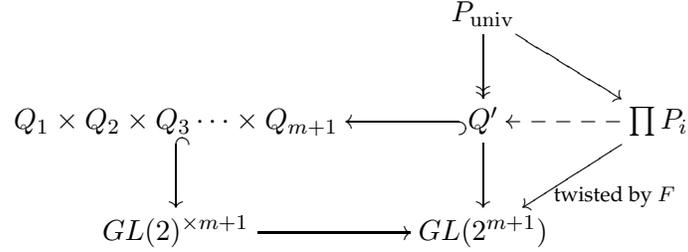
\[\xymatrix{
&P_{\univ} \ar@{->>}[d] \ar[dr]&\\
Q_1\times Q_2\times Q_3 \cdots \times Q_{m+1} \ar@{^{(}->}[d] & Q' \ar@{^{(}->}[l] \ar[d] &\prod P_i \ar[dl]^{\text{twisted by }F} \ar@{-->}[l]\\
GL(2)^{\times m+1} \ar[r] & GL(2^{m+1}) &
}\]

For any $h\in P_{\univ}$, let $(h_1, h_2,h_3, \cdots, h_{m+1})$ be the image of $h$ in $\prod P_i$ and $(g_1, g_2, g_3, \cdots, g_{m+1})$ image in $Q'$. Then $\prod P_i \dashrightarrow Q'$ permutes the factors and sends $(h_1, h_2,h_3, \cdots)$ to $(lh_2l^{-1},\cdots)$. Then $(lh_2l^{-1},\cdots)$ and $(g_1,g_2,g_3, \cdots)$ have the same image under $GL(2)^{\times m+1} \ra GL(2^{m+1})$. Therefore $C_l(h_2)=tg_1$ for some scalar $t\in B(k)$ where $C_l$ is the adjoint action by $l$. In particular, $f_2(h_2)=f_1(g_1)$.The claim is true.
\end{proof}
Then $P_\univ \ra Q_1 \times P_2$ factors through the limit of  
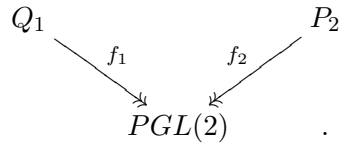
\[\xymatrix{
Q_1 \ar@{->>}[dr]^{f_1}&&P_2 \ar@{->>}[dl]_{f_2}\\
&PGL(2)&.}\]

\begin{Claim} \label{the limit}
the limit of the above diagram is $P_2\times Z(Q_1)=SL(2)\times \m_n \text{ or } SL(2)\times \GG_m$ with \begin{align*}
P_2\times Z(Q_1) & \ra P_2& P_2 \times Z(Q_1) & \ra Q_1\\
(h, k)& \mapsto h& (h,k) & \mapsto (klhl^{-1}) .
\end{align*}
\end{Claim}

\begin{proof} We can prove it directly:  for any $K'$ fitting in the diagram \[\xymatrix{
&K' \ar@{->>}[dl]^{s_1} \ar@{->>}[dr]_{s_2}&\\
Q_1 \ar@{->>}[dr]^{f_1}&&P_2 \ar@{->>}[dl]_{f_2} \ar@{-->}[ll]^{ C_l}\\
&PGL(2)&,
}\] we construct the map 
\begin{align*}
K'&\ra Z(Q_1) \times SL(2) \\
k&\mapsto (s_1(k)C_l(s_2(k))^{-1}, s_2(k)).
\end{align*}
Since the lower triangle is commutative, the map is well defined and obviously it is unique.  \end{proof}
Consider the Tannakian category generated by $\{W_1, V_2\}$. Then it is isomorphic to $\Rep(K_{12})$ for some algebraic group $K_{12}$.  By \ref{the inclusion}, 
\[K_{12}=\im(P_{\univ}\ra \Aut(W_1)\times \Aut(V_2))\subset Q_1\times P_2.\]
Therefore by Claim \ref{the limit}, $K_{12}\subset P_2\times Z(Q_1)=SL(2) \times Z(Q_1)$. 

If $Q_1=GL(2)$, then $\dim K_{12}=4$ and by $GL(2)$ connected, $K_{12}=SL(2)\times \GG_m$. 

If $Q^0_1=SL(2)$ and $Z(Q_1)=\m_n$, then $\dim K_{12}^0=3$ and hence $K_{12}^0=SL(2)$. It suffices to determine the number of the connected components of $K_{12}$. Let $\z$ be a generator of $\m_n$. Then $\z$ and $-\z$ are in the same component of $Q_1$. 
\begin{enumerate}
\item If $n\equiv 0 \pmod 4$,  then $-\z$ is also a generator of $\m_n$. Therefore $K_{12}$ has to be $SL(2)\times \m_n$ to cover the whole $Q_1$.
\item If $n\equiv 2 \pmod 4$, then $\m_n=\pm I \times \m_{\frac{n}{2}}$ and hence $Q_1\cong SL(2)\times \m_\frac{n}{2}$. So besides $SL(2)\times \m_n$,  $K_{12}$ also can be $Q_1$.
\end{enumerate}

In summary, $K_{12}= SL(2)\times \GG_m \text{ or }SL(2)\times \m_k$ for \text{some} $k$. 

Therefore as an irreducible $K-$representation,  $W_i$ is tensor of a $SL(2)-$representation and an irreducible $\m_k$ or $\GG_m$ representation, i.e. $W_i=V_{\s(i)}\otimes L_i$. 

This is the end of the proof of \ref{decomposition of general F}.

\section{Tensor decomposition the Frobenius} \label{tensor decomposition of E}
 \label{tensor decomposition of F}

Now we come back to the context of \ref{main thm 1}. The Dieudonne crystal $\cE=R^1\pi_{\cri\, *}(\cO_X)$ admits the Frobenius map: \[\cE^{\s}\stackrel{F}{\ra} \cE.\] 
 Then we have 
\begin{equation} \label{the frobenius} F: \cV^\s_1 \otimes \cV^\s_2 \otimes \cV^\s_3\cdots \otimes \cV^\s_{m+1}\otimes B(k) \stackrel{\cong}{\ra} \cV_1 \otimes \cV_2 \otimes \cV_3 \cdots \otimes \cV_{m+1} \otimes B(k). \end{equation}
where $B(k)$ is the fractional field of $W(k)$. By \ref{T and iso}, the category of isocrystals over $C$ is Tannakian. 

\begin{prop} \label{for P_i}
For each $i$, $P_i\cong SL(2, B(k))$ and $P_\univ\ra \prod P_i$ is surjective.
\end{prop}
\begin{proof}
Since by \cite[Theorem 6.6]{Berthelot} the crystals on $C/W(k)_\cri$ are exactly vector bundles with a connection over $\tilde{C}$, $\Rep_\C(P_{\univ}\otimes \C)$ is a Tannakian subcategory of $\Rep_\C(G_{\univ})$. By functorality, $G_{\univ}\ra \Aut(E\otimes \C)$ factors through $P_{\univ}\otimes \C$. By \ref{over C} and \ref{MIC} , 
\begin{align*} P\otimes \C =& \im(G_{\univ}\ra \Aut(V)) = \im(SL(2,\C)^{\times m+1}\ra \Aut({\C^2}^{\otimes m+1}))\\ P_i \otimes \C =& \im(G_{\univ}\ra \Aut(\C^2)) = SL(2,\C). \end{align*} 

The group $P_i$ is a $B(k)-$form of $SL(2)$ and admits a faithful two dimensional representation. Therefore $P_i\cong SL(2,B(k))$. 

Therefore  $P=\im(P_{\univ} \ra \prod_i P_i \ra \Aut(E))$ is the same as $\prod_i P_i \ra \Aut(E)$, after tensoring with $\C$. Since it is faithfully flat, it is also true over $B(k)$ and $P=\im(\prod_i P_i \stackrel{\otimes}{\ra} \Aut(E))$. Further, since the kernel of $(\prod P_i\ra \Aut(E))$ is finite,  $\im(P_{\univ} \ra \prod_i P_i)$ is an algebraic subgroup of $\prod_i P_i$ with the same dimension. Since $\prod_i P_i=SL(2,B(k))^{\times m+1}$ are connected, $$\im(P_{\univ} \ra \prod_i P_i)=\prod_i P_i,$$ i.e. $P_{\univ} \ra \prod_i P_i$ is surjective. 
\end{proof}

%\subsection{For $\cV^\s_i \otimes B(k)$}

Now we can interpret isomorphism (\ref{the frobenius}) as follows. We already have a rank $2^{m+1}$ isocrystal admitting a tensor decomposition to $m+1$ rank 2 isocrystals, each corresponding to a standard representation of $SL(2)$. Then for another tensor decomposition to $m+1$ rank 2 isocrystals, just as left hand side of (\ref{the frobenius}), we expect that each component also corresponds to a $SL(2)-$representation which is a corollary of \ref{structure of Q_i}.

\begin{prop} \label{the structure of G}
For each $i$, $Q_i \cong SL(2,B(k))$.
\end{prop}
\begin{proof}

By \ref{for P_i}, $\cV_i$, $\cV^\s_i$ and the isomorphism (\ref{the frobenius}) satisfy the conditions of \ref{structure of Q_i}. Therefore the Tannakian group $Q_i$ corresponds to $\cV^\s_i$ is either $GL(2) \ro SL(2)\times \m_k$.

Furthermore, note $\cV_i$ comes from $\C^2$ in (\ref{the tensor decomposition of V}). Since the local system $\C^2$ on $M$ has a trivial determinant, each isocrystal $\cV_i$ has $\wedge^2 \cV_i=\cO_{\tilde C}$. So correspondingly $\det Q_i=1$ and thus $Q_i = SL(2)$. 
\end{proof}

Apply \ref{decomposition of general F} and note that $W_1$ and $V_2$ are the corresponding objects of $\cV^\s_1\otimes B(k)$ and $\cV_2 \otimes B(k)$ in $\Rep(P_{\univ})$, respectively. We have that there exist a permutation $s\in S_{m+1}$, rank 1 crystals $\cL_i$ with $\otimes_i \cL_i \cong \cO_{\tilde C}$ and isomorphisms 
\[\phi_i: \cV^\s_i \otimes B(k) \ra \cV_{s(i)}\otimes \cL_i \otimes B(k)\]  
such that 
\[F=\otimes_i \phi_i.\] 
In fact, we can refine $\phi_i$ to be a morphism between crystals. 

\begin{prop} \label{W,V} \label{the structure of F}\label{descent to W(k)}
There exist a permutation $s\in S_{m+1}$, rank 1 crystals $\cL_i$ with $\otimes_i \cL_i \cong \cO_{\tilde C}$ and isomorphisms 
\[\phi_i: \cV^\s_i \ra \cV_{s(i)}\otimes \cL_i \]  such that \[F=\otimes_i \phi_i.\]
\end{prop}

\begin{proof}
Since $\cE$ is an $F-$crystal, we still have $F: \otimes \cV^\s_i \ra \otimes \cV_i$. Since each $\phi_i$ is a morphism between effective isocrystals, by \ref{T and iso},  there exists an integer $k_i$ such that $p^{k_i} \phi_i $ is a morphism in $\cris(C)$. We can assume $p^{k_i} \phi_i \neq 0 \pmod p$ at the generic point. Then $p^{-k_1-k_2-\cdots -k_m}\phi_{m+1}$ is also a morphism in $\cris(C)$. In fact, for any $U \subset C$ and $a_{m+1} \in {\cV^\s_{m+1}}(U)$, we can find $a_1 \in {\cV^\s_1}(U), a_2 \in {\cV^\s_2}(U), \cdots$ such that 
\[p^{k_i}\phi_i(a_i) \neq 0 \pmod p\] 
for $1\leq i\leq m$. Then $ p^{-k_1-k_2-\cdots -k_m}\phi_{m+1}(a_{m+1}) \in \cV_{m+1}(U)$. Otherwise, 
\[F(a_1 \otimes a_2\cdots \otimes a_{n+1})=p^{k_1}\phi_1(a_1) \otimes_{B(k)} p^{k_2} \phi_2(a_2) \cdots \otimes_{B(k)} p^{-k_1-k_2-\cdots -k_m}\phi_{m+1}(a_{m+1}) \]
 is not in $\cV_1 \otimes \cV_2 \cdots \otimes \cV_{m+1}(U) $. 

\end{proof}

A straightforward corollary of \ref{descent to W(k)} is that 
\begin{cor}
Viewed as a morphism between crystals, $F$ still preserves pure tensors. 
\end{cor}

%%%%%%%%%%%%%%%%%%%%
%

%

%Now we prove $s(1)=1$ under some mild condition.
%%%%%%%%%%%%%%%%%%%%%
Let $\eta$ be the generic point of $C$ and $\cV_{i, \eta}$ denote the restriction of $\cV_i$ to the crystalline site $\cri(\eta/W(k))$.

Since $C$ parametrizes a family of polarized abelian varieties (with a level structure), it admits a map to the moduli scheme $\cA_{2^m, d, n}\otimes k$. If the image intersects with the ordinary locus in $\cA_{2^m, d, n}\otimes k$, we say ``$C$ intersects the ordinary locus" for simplicity. Note since the ordinary locus is open in $\cA_{2^m, d, n}\otimes k$, the statement is equivalent to the universal family over $C$ is generically ordinary.
Let 
\[0 \ra \o\ra \cE \ra \a \ra 0\] 
be the weight 1 Hodge filtration associated to $\tilde X/\tilde C$. Then from the definition of Mumford curves, especially the action of Hodge group $Q$ on $V$, we know $\o$ is constructed from a line bundle $\scr{L}$ in $\cV_i$ for some $i$, say $i=1$, then 
\begin{equation}\label{the line bundle L} \o\cong \sL \otimes \cV_2 \otimes \cV_3 \cdots \otimes \cV_{m+1} .\end{equation} 
Correspondingly $\a\cong \cV_1/\scr{L} \otimes \cV_2 \otimes \cV_3\cdots \otimes \cV_{m+1}$ and the Hodge filtration of $\cE$ comes from a filtration  $\sL \subset \cV_1$: 
\[\sL \otimes \cV_2 \otimes \cV_3 \cdots \otimes \cV_{m+1}\subset \cE=\cV_1\otimes \cV_2 \otimes \cV_3\cdots \otimes \cV_{m+1} .\]Base change from $W(k)$ to $k$. Denote the reduction of $\tilde{C}$ over $k$ as $C$ and the reduction of $\cE$ as  ${\cE_C}$. Then the Frobenius ${\cE_C}^{(p)} \xrightarrow{ F} {\cE_C}$ factors through ${\a_C}^{(p)}$ and then we have the conjugate spectral sequence:
\begin{equation} \label{conjugate filtration}
0\ra {\a_C}^{(p)} \ra  {\cE_C} \ra  {\o_C}^{(p)}\ra 0.
\end{equation} 

%\begin{lemma}
%There exist an element $s$ in the permutation group $S_3$, rank 1 $F^f-$isocrystal $\cL_i$ with $\otimes \cL_i=\cO_{C}$
%and 
%\[\phi_i: \cV^\s_i \ra \cV_{s(i)}\otimes \cL_i\] for $1\leq i\leq 3$ such that $F=\otimes \phi_i$.
%\end{lemma}

%\begin{cor}
%$\cL^{\otimes 2}_i=\cO_{C}$ for any $i$. 
%\end{cor}

%\begin{proof}
%Since $\cV_i$ corresponds to a $SL(2)-$representation, $\wedge^2 \cV_i$ is trivial. So $\wedge^2 \cV^\s_i \cong (\wedge^2 \cV_i)^\s =\cO_{C}$ as well. Taking the determinant of $\phi_i$ gives $\cL^{\otimes 2}_i=\cO_{C}$.
 
%\end{proof}
\begin{prop} \label{s(j)=j}
If $C$ intersects the ordinary locus, then $s(1)=1$.
\end{prop}

\begin{proof}
Let $c$ be a closed point in the intersection of ordinary locus and $C$. Then restricted to $c$,  consider the composition $F':\cE^{(p)}_C \ra \a_C$ in the following diagram 
\[\xymatrix{ &\o_C \ar[d]&\\
\cE^{(p)}_C \ar@{-->}[dr]^{F'} \ar[r]^{ F_C} &\cE_C \ar[d]^\pi  \\
&\a_C&.
}\] Since $X_c$ is ordinary, $F'_c$ is surjective. 

If $s(1)\neq 1$, Without loss of generality, suppose $s(1)=2$. Note by \ref{descent to W(k)}, 
\begin{align*}
F'(\cE^{(p)}_C)&=\pi\circ  F_C(\cV^{(p)}_1 \otimes\cV^{(p)}_2 \otimes \cV^{(p)}_3 \cdots \otimes \cV^{(p)}_{m+1})_C \\
&=\pi\circ \otimes\phi_i (\cV^{(p)}_1 \otimes\cV^{(p)}_2 \otimes \cV^{(p)}_3 \cdots \otimes \cV^{(p)}_{m+1})_C
\end{align*} 
 From the conjugate spectral sequence, $F_C$ factors through $\a^{(p)}_C$ and $\a=(\cV_1/\scr{L})\otimes {\cV_2} \otimes {\cV_3} \cdots \otimes \cV_{n+1}$, thus $\phi_1: \cV^{(p)}_{1\, C} \ra \cV_{2\, C} \otimes \cL_{1\, C}$ factors through $ \cV^{(p)}_{1\,C}/\scr{L}^{(p)}$, and the image of $\bar\phi_1$ has rank 1. But $\dim_k {\cV_2}_{|c}=2$. So $F'$ can not be surjective. Contradiction.
\end{proof}

Therefore we have 
\begin{equation} \label{phi}
\phi_1:\cV^\s_1 \ra \cV_1 \otimes \cL_1
\end{equation}

%%%%%%%%%%%%%%%%%%%%%%
%\begin{prop} \label{s(1)=1}
%If $C$ intersects ordinary locus, then $s(1)=1$.
%\end{prop}

%\begin{proof}
%Let $c$ be in the locus of the intersection of ordinary locus and $C$. Then restricted to $c$,  consider the composition $F':{\cE_C}^{(p)}\ra{\a_C}$ in the following diagram 
%\[\xymatrix{ &\bar\o \ar[d]&\\
%{\cE_C}^{(p)} \ar@{-->}[dr]^{F'} \ar[r]^{\bar F} &{\cE_C} \ar[d]^\pi  \\
%&\bar\a&.
%}\] Since $\tilde X_c$ is ordinary, $F'_c$ is surjective. 

%If $s(1)\neq 1$, Without loss of generality, suppose %$s(1)=2$. Note by \ref{descent to W(k)}, 
%\begin{align*}
%F'({\cE_C}^{(p)})&=\pi\circ \bar F(\bar\cV^{(p)}_1 \otimes\bar \cV^{(p)}_2 \otimes \bar\cV^{(p)}_3 \cdots \otimes \bar \cV^{(p)}_{m+1}) \\
%&=\pi\circ \otimes\bar \phi_i (\bar\cV^{(p)}_1 \otimes\bar \cV^{(p)}_2 \otimes\bar \cV^{(p)}_3 \cdots \otimes \bar \cV^{(p)}_{m+1})
%\end{align*} 
% From the conjugate spectral sequence, $\bar F$ factors through ${\a_C}^{(p)}$ and $\a=(\cV_1/\scr{L})\otimes {\cV_2} \otimes {\cV_3} \cdots \otimes \cV_{n+1}$, thus $\bar\phi_1:\bar \cV^{(p)}_1 \ra \bar \cV_2 \otimes \bar \cL_1$ factors through $\bar \cV^{(p)}_1/\scr{L}^{(p)}$, and the image of $\bar\phi_1$ has rank 1. But $\dim_k {\cV_2}_{|c}=2$. So $F'$ can not be surjective. 
%\end{proof}
%%%%%%%%%%%%%%%%%%%%%%%

\begin{rmk} \label{intersects with ordinary}
The Mumford curve $M$ is defined over the reflex field $K$, and let $\frak{p}$ be the prime of $K$ over $p$. 

Let $r=[K_\frak{p}: \Q_p]$. Then by \cite[Theorem 1.2]{Zuo2}, there are two Newton polynomials in $C/k$, it is either $\{2^{m+1+\e(D)}\times \frac{1}{2}\}$ or $\{2^{m+1-r+\e(D)}\times 0, 2^{m+1-r+\e(D)}. \binom{r}{i} \times \frac{i}{r}\cdots, 2^{m+1-r+\e(D)}\times 1\}$. So $C$ intersects with ordinary locus if and only if $r=1$.

So there are infinitely many prime $p$ over which the reduction of Mumford curve at $p$ is generically ordinary.
\end{rmk}

\section{The surjectivity of $\s^*-\mathrm{Id}$ on the Picard group} \label{the surjectivity}

Our purpose is to construct a rank 2 Dieudonne crystal in the tensor decomposition of $\cE$. We already have \[\phi_1: \cV_1 \ra \cV_1 \otimes \cL_1.\] So it only remains to ``eliminate" $\cL_1$. We can achieve this goal in next section and the key ingredient is \ref{surjectivity} which we will prove in this section. 

Let $\s$ be the absolute Frobenius of $ C/k$ and $\Pic(C/W(k)_{\cri})$ denote the group of the rank 1 crystals on $C$. The following general principle guarantees $\cL_1=\Pp(\sL')$.

\begin{prop} \label{surjectivity}
The group endomorphism $\s^*-\mathrm{Id}$ of $\Pic(C/W(k)_\cri)$ is surjective. 
\end{prop}

\subsection{The proof }

\begin{lemma} \label{W(k)} $\s^*-\mathrm{Id}$ acts on $W(k)$ surjectively. \end{lemma}

\begin{proof}
Since $k$ is algebraically closed, $\Pp$ acts on $k$ surjectively. Then for any $b\in W(k)$, we can find \[\Pp(a_0)=b+pb_1\] and there exists $a_1, a_2, a_3, \cdots$ such that 
\[\Pp(a_1)=b_1 + pb_2,\] \[\Pp(a_2)=b_2+ pb_3,\] \[\Pp(a_3)=b_3+pb_4\cdots\]
Then $\Pp(\sum_i p^i a_i)=b$. In fact, since $W(k)$ is p-adically complete, $\sum p^i a_i \in W(k)$ and $\Pp(\sum p^i a_i )-b$ is contained in $p^n W(k)$ for any $n$. Therefore $\Pp(a+\sum p^i a_i )-b=0$.
\end{proof}

Now we recall the definition of Atiyah class. For a more detailed explanation, we refer the reader to \cite[10.1]{Dan}. 

Let $\scr{I}$ be the ideal sheaf of the diagonal set of $\tilde X\times \tilde X$ and $\cO_{2\D}=\cO_{\tilde X\times \tilde X}/ \scr{I}^2 $.
\begin{defn}
For any smooth proper variety $\tilde X$ and vector bundle $V$ over $\tilde X$, the Atiyah class is the extension class of \[0\ra V\otimes \O^1_{\tilde X} \ra {p_1}_*(p^*_2 V \otimes \cO_{2\D}) \ra V \ra 0.\]
\end{defn}
Atiyah class is the unique obstruction to the existence of a connection on $V$.

By \cite[Remark 3.7]{Atiyah class}, the Atiyah class of any line bundle coincides with its first Chern class. So  line bundles with a connection over a curve are exactly those of degree 0. 
\begin{lemma} \label{restriction}
The  restriction of $\s^*-\mathrm{Id}$ to $\mathrm{Pic}^0(C/k_\cri)$ is surjective. 
\end{lemma}
\begin{proof} Note the rank 1 crystal on the site $C/k_\cri$ is equivalent to a line bundle on $C/k$ with connection. For any $\sL\in \Pic (C/k_\cri)$, $\s^* (\sL)=\sL^p$. So it suffices to show that for any degree 0 line bundle with connection $(\sL, \nabla)$, there exists a line bundle with connection $(L, \nabla_L)$ such that \[(L, \nabla_L)^{p-1}\cong (\sL, \nabla).\]

Since $k$ is algebraically closed, the Jacobian $\mathrm{Jac}(C/k)$ is a divisible group. Therefore we always can find a line bundle $L\in \mathrm{Jac}(C/k)$ such that $L^{p-1}\cong \sL$. 

Note the set of connections of $\sL$  is a torsor under 
\[\Hom(\sL, \sL\otimes \o_{C})=\G(\o_{C})=k^g.\] The same for $L$. For any two connections $\nabla_L, \nabla'_L$ on $L$, let $\na_L-\na'_L=h\in \G(\o_{C})$. 

Thus to find the connection $\nabla_L$, it suffices to show the $(p-1)-$th power is an \textsl{injection} from the connections on $L$ to the connections on $\sL$. Then for any local section $\otimes_i s_i$,  
\begin{align*}&((\na_L+h)^{p-1}-\nabla^{p-1}_L)(\otimes^{p-1}_{i=1} s_i)\\
=& \sum^{p-1}_{i=1} \cdots \otimes h.s_i\otimes \cdots \\
=& (p-1) (\prod_i s_i) h (1\otimes 1\otimes \cdots \otimes 1). \end{align*} 

Therefore for any connection $\nabla_\sL$, if $g=\nabla_\sL -\nabla^{p-1}_L$, then $(\nabla_L+\frac{g}{p-1})^{p-1}=\nabla_\sL$. So $\Pp$ acts on $\Pic^0(C/k_\cri)$ surjectively. \end{proof}

\begin{lemma} \label{surjective on kernel}
$\s^*-\mathrm{Id}$ maps $H^1(C/W(k)_\cri, \cO_{\tilde C})$ to itself surjectively.
\end{lemma}
%%%%%%%%%%%%%%%%%%
%First we quote a result from (\cite[Page 143]{AV}): 
%\begin{prop} \label{from AV}
%Let $k$ be an algebraically closed field with characteristic p and  $V$ be any vector space with a $p-$linear map $x\mapsto x^{(p)}$, there is a unique decomposition, invariant under $x\mapsto x^{(p)}$: 
%\[V=V_s \oplus V_n,\] such that $V_s$ has a basis $x_1,\cdots, x_k$ for which $x^{(p)}_i=x_i$, and such that $x\mapsto x^{(p)}$ is a nilpotent map on $V_n$.
%%\end{prop}
%Then we can prove our \ref{surjective on kernel}: 
%%%%%%%%%%%%%%%%
\begin{proof}
By comparison theorem,
\[H^1(C/W(k)_\cri, \cO_{\tilde C}) \cong \mathbb{H}^1(\tilde C, \O^._{\tilde C})\cong W(k)^{2g}.\] Let $N$ denote the free $W(k)-$module with $\s^*$ action. Then $V:=N/pN$ is a $k-$vector space with $p-$linear action. By a result in (\cite[Page 143]{AV}), \[V=V_s\oplus V_n \] where $V_s$ is the semisimple part and $V_n$ the nilpotent part. 
On $V_n$, since $\s^*$ acts nilpotently, $\Pp$ is invertible and hence surjective. On $V_s$, by \ref{W(k)}, we can find $\l$ such that $\Pp(\l)=1$. Then for each $k$, $\Pp(\l x_k)=x_k$. Therefore $\Pp$ acts on $V$ surjectively. 

Back to $N$, for any $b\in N$, we can choose $a_0$ such that \[\Pp(a_0)=b+pb_1.\]Then choose $a_1$ such that \[\Pp(a_1)=b_1+pb_2.\] Following this way, we can find $a_2, a_3, \cdots$.
Similar to the proof of \ref{W(k)}, we have 
\[\Pp(a_0+pa_1+p^2a_2+ \cdots+p^na_n + \cdots)=b.\]  \end{proof}

Now we can prove \ref{surjectivity}:
\begin{proof}
Note $\Pic(C/W(k)_\cri)\cong H^1(C/W(k)_\cri, \cO^*_{C})$. We have the sequence \[0\ra (1+p\cO_{C})^* \ra \cO^*_{C} \ra (\cO_{C}/p)^* \ra 0\] and $\Pp$ acts on the long exact sequence.
Since $\cha k>2$,  the exponential and logarithm maps converge and thus give an isomorphism between abelian groups 
\[\cO_{C} \cong (1+p\cO_{C})^*.\] 
So the cohomology groups are isomorphic: 
\[H^1(C/W(k)_\cri, \cO_C) \cong H^1(C/W(k)_\cri, (1+p\cO_C)^*).\] 
We have the long exact sequence 
\begin{align*}
&H^1(C/W(k)_\cri, \cO_C) \ra H^1(C/W(k)_\cri, \cO^*_C) \xrightarrow{g}& \\
& H^1(C/W(k)_\cri, (\cO_C/p)^*)  \cong  \Pic(C/k_\cri)  \ra H^2(C/W(k)_\cri, \cO_C)& .
\end{align*}
By \cite[Theorem 6.6]{Berthelot} ,  the category of crystals on $C$ is equivalent to the category of vector bundles with a connection on $\tilde C$. Therefore $\Pic(C/W(k)_\cri)$ is isomorphic to the group of line bundles with a connection on $\tilde C$ and $g$ is the pull back of such line bundle from $\tilde C$ to $C$. Therefore $\im g\subset \Pic^0(C/k_\cri)$. 

Since the obstruction to deform the line bundle from $C$ to $\tilde C$ vanishes and the deformation preserves the degree, $\Pic^0(\tilde C) \ra \Pic^0(C)$ is surjective. In fact, for any degree 0 line bundle $\sL$ on $\tilde C$, it corresponds to a divisor $\sum_i n_ip_i$ with each $p_i$ a $k-$point. Then by Hensel's lemma, each $p_i$ lifts to a $W(k)-$point $\tilde p_i$( though not uniquely). Let $\sum_i n_i\tilde p_i=\tilde \sL \in \Pic^0(\tilde C)$ and then $\tilde \sL$ reduces to $\sL$.

For the connection, for any $(\sL, \nabla)\in \Pic^0(C/k_\cri)$, choose a lifting $\tilde \sL\in \Pic^0(\tilde C)$ of $\sL$ and a connection $\tilde \nabla$ on $\tilde \sL$. Let $\nabla'$ be the reduction of $\tilde \nabla$, then $\nabla'-\nabla=f \in \G(\o_{C})$. Choose $\tilde f\in \G(\o_{\tilde C})$ such that $\tilde f$ reduces to $f$. Then $\tilde \nabla-\tilde f$ reduces to $\nabla$. And $g(\tilde \sL, \tilde \nabla-\tilde f)=(\sL, \nabla)$. 

Therefore $\im g=\Pic^0(C/k_\cri)$. So we have the following sequence:
\[H^1(C/W(k)_\cri, \cO_{C}) \ra H^1(C/W(k)_\cri, \cO^*_{C}) \ra \Pic^0(C/k_\cri) \ra 0\]

By \ref{restriction} and \ref{surjective on kernel}, $\s^*-\mathrm{Id}$ induces surjective endomorphisms on $H^1(C/W(k)_\cri, \cO_{C}) $ and $ \Pic^0(C/k_\cri)$. Therefore $\Pp$ maps $H^1(C/W(k)_\cri, \cO^*_{C})$ surjectively on itself. 
% need to chase the diagram to prove it.
\end{proof}

\begin{rmk} \label{p>2}
In the proof of \ref{surjectivity}, we use the convergence of exponential and logarithm, which are true if and only if the characteristic $p>2$.
\end{rmk}

\section{The Dieudonne crystal $\cV$ and the unit crystal $\cT$}
\label{the crystals}

Now by \ref{surjectivity} we can choose $\sL' \in \Pic(C/W(k)_\cri)$ such that $\Pp(\sL')=\cL^{-1}_1$(for $\cL_1$ see \ref{phi}). Then $\phi_1$ induces an isomorphism 
\begin{equation} \label{gamma}
 \g: \cV^\sigma_1\otimes \sL'^\sigma \otimes B(k)\ra \cV_1\otimes \sL'\otimes B(k) .  
\end{equation}
Let $\cV=\cV_1\otimes \sL$. 
Similarly, we have the isomorphism 
\[\b: \cV^\sigma_2\otimes \cV^\sigma_3 \cdots \otimes \cV^\s_{m+1}\otimes (\sL'^{-1})^ \sigma \otimes B(k)\ra \cV_2 \otimes \cV_3 \cdots \otimes \cV_{m+1} \otimes \sL'^{-1}\otimes B(k).\] Denote $\cV_2 \otimes \cV_3 \cdots \otimes \cV_{m+1} \otimes \sL'^{-1}$ as $\cT$. Therefore as crystals, 
\[\cE \cong \cV\otimes \cT\]  and as a morphism between crystals $F= \g\otimes \b$. Then $ V=pF^{-1}=p\g^{-1} \otimes \b^{-1}.$ 

\begin{lemma}
The morphism $\b: \cT^\s \ra \cT$ is an isomorphism between crystals.
\end{lemma}
\begin{proof}
We have known that $\g\neq 0 \pmod p$. Over $C$, \ref{conjugate filtration}shows the Frobenius \[ F_C=\g_C\otimes \b_C: {\cE_C}^\sigma \ra {\cE_C}\] induces an injection
\[\a^{(p)}_C=(\cV_{1}/ \sL \otimes \bar \sL' )^{(p)}_C\otimes \cT^{(p)}_C \ra {\cE_C}\cong \cV_C \otimes \cT_C.\] Therefore $\b_C$ is an isomorphism between $\cT^\s_C$ and $\cT_C$. 

Note the fact that for any $W(k)-$algebra $R$ and any $r\in R$, if the image $\bar r \in \bar R$ over $k$ is a unit, then $r$ is a unit in $R$. So $\b$ is an isomorphism between crystals $\cT^\s$ and $\cT$.
\end{proof}

Then $\b^{-1}$ is also a morphism between crystals. Since $ V=pF^{-1}=p\g^{-1} \otimes \b^{-1}$,  so is $p\g^{-1}$. Therefore $F_\cV:=\g$ and $V_\cV:=\g^{-1}$ can serve as Frobenius and Verschiebung of $\cV$, which makes $\cV$ a Dieudonne crystal. The fact that $p^{-k'} \b^{-1}$ and $p^{-k} \b$ are isomorphisms implies $\cT$ is a unit root crystal. We have the following summary.

\begin{cor} \label{decomp of F}  \label{the filtration}
\[(\cV, F_{\cV}=p^k\g, V_{\cV}=p^{k'+1} \g^{-1})\] is a Dieudonne crystal, \[(\cT, F_{\cT}=p^{-k} \b)\] is a unit root crystal and 
\[(\cE,F) \cong (\cV, F_\cV) \otimes (\cT, F_\cT).\]
The Hodge filtration of $\cE$ comes from a sub line bundle $\sL \otimes \sL'$ of $\cV$. 
\end{cor} Let the filtration $\rm{Fil}_\cV$ be $\sL\otimes \sL' \subset \cV$ and $\rm{Fil}_\cT$ be the trivial filtration.

Now we switch to BT groups.

\section{The BT groups corresponding to $\cV$, $\cT$ and $\cE$} \label{the BT group}

From \cite[Main Theorem 1]{de J}, we know that over a smooth curve $C/k$, the category of finite locally free Dieudonne crystals on $\cri(C/W(k))$ is equivalent to the category of BT groups on $C$. Obviously $(\cE,F,V)$ corresponds to $X[p^\infty]$. Let $G$ be the BT group over $C$ corresponding to $(\cV, F_\cV, V_\cV)$.  From \cite{de J}, we know the BT group $\bG$ induces a filtration of $\DD(\bG)_C=\cV_C$:
\begin{equation} \label{the filtration from bG} 0\ra \o_{G} \ra \cV_C \ra t_{G^*} \ra 0.\end{equation}

\begin{lemma}\label{same filtration}
The above filtration \ref{the filtration from bG} coincides with the filtration $\mathrm{Fil}_\cV \pmod p$.
\end{lemma}

\begin{proof} 
From \ref{decomp of F}, 
$\ker  F_{\cV \, C} = (\sL \otimes \sL')^{(p)}_C.$ By \cite[Theorem 2.5.2 and Remark 2.5.5]{deJ}, the subbundle of $\cV$ satisfying this condition is unique and  $\o_G \cong  \sL \otimes \sL'_C$. 
\end{proof}
Then the filtration $\rm{Fil}_\cV$ is just
\[0 \ra \o_{{G}} \ra \cV_C \ra t_{\bar{G^*}} \ra 0 .\]
Note $\cV_C$ admits a connection $\nabla: \cV_C \ra \cV_C \otimes \O_C$. The connection and the filtration induce the Higgs field: $\theta_{{G}}: \o_{{G}} \ra t_{{G}}\otimes \O_C$.
\[\xymatrix{
0 \ar[r] & \o_{{G}} \ar@/_/[rrd] \ar[r]& \cV_C \ar[d]^{\bar\nabla }\\
 &&\cV_C \otimes \O^1_{C} \ar[r] & t_{\bar{G^*}}\otimes \O^1_{C} \ar[r]& 0.
}\]

Since $\cT$ is a unit root crystal, by \cite[2.4.10]{BM}, $\cT$ comes from an etale BT group $\bH$ over $C$. In particular, $\DD(\bH[p^n])\cong \cT/p^n$ and each truncated $\cT/p^n$ comes from a local system \cite[Theorem 2.2]{Crew}
\[\r_n: \pi_1(C, c) \ra  GL(4, \Z/p^n) . \]
Then there exists a finite etale covering $f_n: C' \ra C$ such that $\pi_1(C', c)\cong \ker \r_n$. Therefore we have $f^*_n(\cT/p^n)\cong \cO_{C/W_n}^{\oplus m}$ as unit root $F-$crystals.  By \cite[2.4.1]{BM}, over smooth curve $C$, the category of finite locally free etale group schemes is equivalent to the category of $p-$torsion unit crystals.  Thus 
\begin{equation} \label{triviality} f^*_n(\bH[p^n])\cong (\Z/p^n)^{\oplus m}. \end{equation}

\begin{defn}
Define the binary operation between two BT groups: \[G\otimes H:=\colim_n (G[p^n]\otimes_\Z H[p^n]).\]
\end{defn}

\begin{rmk}
The inductive system $(G[p^n]\otimes_\Z H[p^n])$ is explained in \cite{Xia}. Note in general $G\otimes H$ is just an abelian sheaf rather than a group scheme. But in our case, $H$ is etale and  $G\otimes H$ is indeed a BT group and $(G\otimes H)[p^n]=G[p^n]\otimes_\Z H[p^n]$.
\end{rmk}

\begin{prop} \label{first order}
\[X[p^n]\cong \bGn .\]
\end{prop}

\begin{proof}
We will show that $\mathbb{D}(\bGn)=\cV \otimes \cT/p^n=\cE/p^n$( \ref{decomp of F}) as Dieudonne crystals. Over $C'$, 
\[\tag{*}\DD_{C'}(f^*_n(\bGn))\cong f^*_n(\cV/p^n)^{\oplus m}\cong f^*_n \cV/p^n \otimes_{\cO_{\tilde C}} f^* \cT/p^n\] as Dieudonne crystals.
Both sides have effective descent datum with respect to $C' \ra C$. For any $g\in \mathrm{\Aut}(C'/C)$, $g^*$ acts on both of $f^*(\bGn)$ and $f^*_n \cV/p^n \otimes_{\cO_{\tilde C}} f^* \cT/p^n$ which is compatible with the functor $\DD_{C'}$: 
\[\xymatrix{
f^*(\bGn)\ar[r]^{g^*} \ar[d]^{\DD_{C'}}& f^*(\bGn) \ar[d]^{\DD_{C'}} \\
f^*(\cV \otimes \cT/p^n) \ar[r]^{g^*} & f^*(\cV \otimes \cT/p^n)
}
\] is commutative( we leave the details leave to the reader). Therefore the isomorphism (*) between effective descent datum also descends to $C$. 

Then we have 
\[\DD_{C}(\bGn)=(\cV \otimes \cT/p^n, F_\cV \otimes F_\cT, V_\cV \otimes F^{-1}_\cT).\] 

Since $C$ is smooth over an algebraically closed field $k$, it has locally $p$-basis. Therefore we can apply (\cite{BM}, 4.1.1), the Dieudonne functor is fully faithful, so 
\[\bGn\cong X[p^n].\]
\end{proof}

\begin{cor}
$\bG\otimes H=X[p^\infty] .$
\end{cor}

To complete the proof of \ref{main thm 1}, it remains to show the isomorphism in \ref{first order} lifts to $\tilde C$.

\begin{prop}
The curve $C$ is a versal deformation of the BT group $G$. 
\end{prop}

\begin{proof} From (\cite[Theorem 0.9]{Moller}), any Shimura curve of Hodge type admits the maximal Higgs field. So $\tilde C$ and thus $C$ has maximal Higgs field:
\[\xymatrix{
0 \ar[r] & {\o_C} \ar@/_/[rrd] \ar[r]& \cE_C \ar[d]^{\bar\nabla }\\
 &&\cE_C \otimes \O^1_{C} \ar[r] &{\a_C}\otimes \O^1_{C} \ar[r]& 0
}\] the induce map $\theta: {\o_C} \ra {\a_C} \otimes \O^1_C$ is an isomorphism. 

By \ref{the filtration}, the filtration of $\cE_C$ comes from that of $\cV_C$. So $\theta=\theta_{{G}} \otimes \mathrm{Id}_\cT$. Therefore the Higgs field of $\cV_C/C$ is maximal and combining with \ref{same filtration}, it implies $\o_{G} \cong t_{G^*} \otimes \O^1_{C}$. By (\cite[A.2.3.6]{Ill}), $C$ is a versal deformation of the BT group $G$. 
\end{proof}

From (\cite[Theorem 1.1]{Xia}), such a curve $C$ admits a lifting to $W(k)$ over which $\bG$ admits a lifting as a BT group. 

\begin{prop} \label{lift along C}
The lifting of $C$ coincides with $\tilde C$. 
\end{prop}

\begin{proof}
From (\cite[ V, Theorem 1.6]{Messing}), it is known that the lifting of the BT group $G$ is equivalent to lifting the filtration $\o_{G} \hookrightarrow \cV_C$. By \ref{same filtration}, it is equivalent to lifting $\sL \otimes \sL' \hookrightarrow \cV_C$ and hence the curve $\tilde C/W(k)$ admits a lifting of $G$. From \cite[Theorem 1.1]{Xia}, we know the lifting of $G$ is unique. 
\end{proof}

Let $\tilde G$ be the lifting of $G$ on $\tilde C$ and $\tilde H$ be the lifting of $H$ on $\tilde C$. Since $H$ is etale, $\tilde H$ is etale and unique up to isomorphism. By Proposition A.3 in \cite{Xia}, $\tilde G \otimes \tilde H$ is a BT group.  

\begin{prop}
\[\tilde X[p^\infty]\cong \tilde G \otimes {\tilde H}.\]
\end{prop}

\begin{proof}
In \ref{first order}, we have shown that $X[p^\infty] \cong G \otimes H$ as BT groups over $C$. Both sides are liftable to $\tilde C$( \ref{lift along C}), induced by the same filtration $\o_{\tilde G} \otimes \cT \hookrightarrow \cE$(\ref{same filtration}). Again by  \cite[V Theorem 1.6]{Messing}, $\tilde X[p^\infty]\cong \tilde G \otimes \tilde H$. 
\end{proof}

Now the proof of \ref{main thm 1} is complete. 
%%%%%%%%%%%%%%%%%
% In the case of generically ordinary, the special Newton polygon may imply that the absolute Frobenius fixes one piece. Furthermore, distinct permutations may give all possible Newton polygon for the Mumford curve.
%%%%%%%%%%%%%%%%%

%%%%%%%%%%%%%%%%%
% We want to show $F_1: \cV^{(p)}_1 \ra \cV_1$ is well defined. But we only know $\cV^{(p)}_1 \ra \cV_1 \otimes \cL_1$. Take the determinant and it gives the square of $\cL_1$ is trivial. Over $\C$, it is a monodromy with eigenvalues $\pm 1$. Then it is trivial after pulling back to a order two covering $\tilde C'\ra \tilde C$. \tilde Can it tell us more?
%%%%%%%%%%%%%%%%%

\bibliographystyle{amsplain}
\bibliography{mybib}{}

\providecommand{\bysame}{\leavevmode\hbox to3em{\hrulefill}\thinspace}
\providecommand{\MR}{\relax\ifhmode\unskip\space\fi MR }
\providecommand{\MRhref}[2]{%
  \href{http://www.ams.org/mathscinet-getitem?mr=#1}{#2}
}
\providecommand{\href}[2]{#2}
\begin{thebibliography}{10}

\bibitem{Andre}
Yves Andr{\'e}, \emph{Mumford-{T}ate groups of mixed {H}odge structures and the
  theorem of the fixed part}, Compositio Math. \textbf{82} (1992), no.~1,
  1--24. \MR{1154159 (93b:14026)}

\bibitem{BM}
Pierre Berthelot and William Messing, \emph{Th\'eorie de {D}ieudonn\'e
  cristalline. {III}. {T}h\'eor\`emes d'\'equivalence et de pleine
  fid\'elit\'e},  \textbf{86} (1990), 173--247. \MR{1086886 (92h:14012)}

\bibitem{Berthelot}
Pierre Berthelot and Arthur Ogus, \emph{Notes on crystalline cohomology},
  Princeton University Press, Princeton, N.J., 1978. \MR{0491705 (58 \#10908)}

\bibitem{Crew}
Richard Crew, \emph{{$F$}-isocrystals and {$p$}-adic representations},
  Algebraic geometry, {B}owdoin, 1985 ({B}runswick, {M}aine, 1985), Proc.
  Sympos. Pure Math., vol.~46, Amer. Math. Soc., Providence, RI, 1987,
  pp.~111--138. \MR{927977 (89c:14024)}

\bibitem{StacksProject}
A.~J. de~Jong, \emph{Stacks project}, http://stacks.math.columbia.edu/.

\bibitem{deJ}
\bysame, \emph{Crystalline {D}ieudonn\'e module theory via formal and rigid
  geometry}, Inst. Hautes \'Etudes Sci. Publ. Math. (1995), no.~82, 5--96
  (1996). \MR{1383213 (97f:14047)}

\bibitem{Deli}
P.~Deligne, \emph{Cat\'egories tannakiennes}, The {G}rothendieck {F}estschrift,
  {V}ol.\ {II}, Progr. Math., vol.~87, Birkh\"auser Boston, Boston, MA, 1990,
  pp.~111--195. \MR{1106898 (92d:14002)}

\bibitem{Dan}
Daniel Huybrechts and Manfred Lehn, \emph{The geometry of moduli spaces of
  sheaves}, Aspects of Mathematics, E31, Friedr. Vieweg \& Sohn, Braunschweig,
  1997. \MR{1450870 (98g:14012)}

\bibitem{Ill}
Luc Illusie, \emph{D\'eformations de groupes de {B}arsotti-{T}ate (d'apr\`es
  {A}. {G}rothendieck)}, Ast\'erisque (1985), no.~127, 151--198, Seminar on
  arithmetic bundles: the Mordell conjecture (Paris, 1983/84). \MR{801922}

\bibitem{Lefschetz}
Solomon Lefschetz, \emph{Algebraic geometry}, Princeton University Press,
  Princeton, N. J., 1953. \MR{0056950 (15,150h)}

\bibitem{Zuo2}
Sheng Mao and Kang Zuo, \emph{On the {N}ewton polygons of abelian varieties of
  {M}umford's type}, arXiv:1106.3505v1.

\bibitem{Messing}
William Messing, \emph{The crystals associated to {B}arsotti-{T}ate groups:
  with applications to abelian schemes}, Lecture Notes in Mathematics, Vol.
  264, Springer-Verlag, Berlin, 1972. \MR{0347836 (50 \#337)}

\bibitem{Moller}
Martin M{\"o}ller, Eckart Viehweg, and Kang Zuo, \emph{Special families of
  curves, of abelian varieties, and of certain minimal manifolds over curves},
  Global aspects of complex geometry, Springer, Berlin, 2006, pp.~417--450.
  \MR{2264111 (2007k:14054)}

\bibitem{Mum}
D.~Mumford, \emph{A note of {S}himura's paper ``{D}iscontinuous groups and
  abelian varieties''}, Math. Ann. \textbf{181} (1969), 345--351. \MR{0248146
  (40 \#1400)}

\bibitem{famofav}
David Mumford, \emph{Families of abelian varieties}, Algebraic {G}roups and
  {D}iscontinuous {S}ubgroups ({P}roc. {S}ympos. {P}ure {M}ath., {B}oulder,
  {C}olo., 1965), Amer. Math. Soc., Providence, R.I., 1966, pp.~347--351.
  \MR{0206003 (34 \#5828)}

\bibitem{AV}
\bysame, \emph{Abelian varieties}, Tata Institute of Fundamental Research
  Studies in Mathematics, vol.~5, Published for the Tata Institute of
  Fundamental Research, Bombay, 2008, With appendices by C. P. Ramanujam and
  Yuri Manin, Corrected reprint of the second (1974) edition. \MR{2514037
  (2010e:14040)}

\bibitem{Rutger}
Rutger Noot, \emph{Abelian varieties with {$l$}-adic {G}alois representation of
  {M}umford's type}, J. Reine Angew. Math. \textbf{519} (2000), 155--169.
  \MR{1739726 (2001k:11112)}

\bibitem{Atiyahclass}
Th. Peternell, J.~Le~Potier, and M.~Schneider, \emph{Direct images of sheaves
  of differentials and the {A}tiyah class}, Math. Z. \textbf{196} (1987),
  no.~1, 75--85. \MR{907410 (88m:32057)}

\bibitem{Saa}
Neantro Saavedra, \emph{Fibered categories and systems of categories}, Univ.
  Nac. Ingen. Inst. Mat. Puras Apl. Notas Mat. \textbf{4} (1966), 109--117.
  \MR{0236242 (38 \#4539)}

\bibitem{Zuo}
Eckart Viehweg and Kang Zuo, \emph{A characterization of certain {S}himura
  curves in the moduli stack of abelian varieties}, J. Differential Geom.
  \textbf{66} (2004), no.~2, 233--287. \MR{2106125 (2006a:14015)}

\bibitem{Xia}
J~Xia, \emph{On the deformation of {B}arsotti-{T}ate group over a curve},
  arXiv:1303.2954.

\end{thebibliography}

\end{document}